\documentclass[11pt]{amsart}
\usepackage{color}
\usepackage{graphicx}
\usepackage{microtype}
\usepackage{hyperref}

\newtheorem{thm}{Theorem}
\newtheorem{lem}[thm]{Lemma}
\newtheorem{prop}[thm]{Proposition}

\theoremstyle{definition}
\newtheorem{defn}[thm]{Definition}
\newtheorem{ex}[thm]{Example}

\newcommand\arxiv[1]{\href{http://arxiv.org/abs/#1}{arXiv:#1}}
\newcommand\cA{\mathcal A}
\newcommand\cD{\mathcal D}
\newcommand\defi[1]{\emph{#1}}
\newcommand\degree{\operatorname{deg}}

\newcommand\isom{\cong}

\newcommand\RR{\mathbb R}
\newcommand\Spec{\operatorname{Spec}}
\newcommand\X{\mathfrak X}
\newcommand\ZZ{\mathbb Z}

\title{Excluded homeomorphism types for dual complexes of surfaces}
\author{Dustin Cartwright}
\address{Department of Mathematics \\ University of Tennessee \\
          227 Ayres Hall \\ Knoxville, TN 37996-1320}
\email{cartwright@utk.edu}

\begin{document}
\begin{abstract}
We study an obstruction to prescribing the dual complex of a strict semistable
degeneration of an algebraic surface. In particular, we show that if $\Delta$
is a complex homeomorphic to a $2$-dimensional manifold with negative
Euler characteristic, then $\Delta$ is not the dual complex of any
semistable degeneration. In fact, our theorem is somewhat more general and
applies to some complexes homotopy equivalent to such a manifold. Our
obstruction is provided by the theory of tropical complexes.
\end{abstract}

\maketitle

The dual complex of a semistable degeneration is a combinatorial encoding of the
combinatorics of the components of the special fiber. In recent years, it has
been studied because of connections to tropical geometry~\cite{helm-katz},
non-Archimedean analytic geometry~\cite{berkovich}, and birational
geometry~\cite{dfkx,brown-foster}. In this paper, we study obstructions to
realizing arbitrary complexes as dual complexes of degenerations of surfaces.

We let $R$ be any rank~$1$ valuation ring with algebraically closed residue
field and we will consider a \defi{degeneration} over~$R$ to be a flat, proper
scheme~$\X$ over $\Spec R$ which is strictly semistable in the sense
of~\cite[Sec.~3]{gubler-rabinoff-werner}. Specifically, we require that $\X$ is
covered by open sets which admit \'etale morphisms over~$R$ to $\Spec R[x_0,
\ldots, x_n] / \langle x_0 \cdots x_m - \pi \rangle$ for some $m \leq n$ and
some $\pi$ in the maximal ideal of~$R$. The dual complex of~$\X$ is a
$\Delta$-complex with one vertex for each irreducible component of the special
fiber and higher-dimensional simplices for each connected component where
irreducible components intersect.

Since semistability implies that the special fiber is normal crossing, the
dimension of the dual
complex is at most the relative dimension of the family~$\X$. In dimension~$1$, any
graph is the dual complex of some degeneration of
curves~\cite[Cor.~B.3]{baker}. However, in this paper, we show that the
analogous statement is not true in dimension~$2$. 
\begin{thm}\label{thm:main}
There is no strict semistable degeneration over a rank~$1$ valuation ring~$R$,
whose general fiber is a smooth, geometrically irreducible surface and such that
the dual complex of the special fiber is homeomorphic to a topological
surface~$\Sigma$ with $\chi(\Sigma) < 0$.
\end{thm}

We conjecture that Theorem~\ref{thm:main} can be strengthened to replace
``homeomorphic'' with ``homotopy equivalent.'' In fact, we can prove a
strengthening in this direction which applies to $\Delta$-complexes formed from a
manifold with negative Euler characteristic by attaching additional simplices
in a controlled way. First, what we call ``fins'' are allowed so long as they
don't change the homotopy type and where the gluing is along a subset that's not
too complicated. Second, arbitrary complexes may be attached to the manifold,
so long as the gluing is along a finite set. These complexes are collectively
the ``ornaments'' in the following definition.

\begin{defn}\label{def:manifold-fins-ornaments}
We say that a $2$-dimensional $\Delta$-complex~$\Delta$ is a \defi{manifold with fins and ornaments} if there exists a subcomplex $\Sigma$, the \defi{manifold}, subcomplexes $F_1, \ldots, F_n$, the \defi{fins}, and a
subcomplex~$O$, the \defi{ornaments}, such that:
\begin{enumerate}
\item We have a decomposition $\Delta = \Sigma \cup F_1 \cup \ldots \cup F_n
\cup O$.
\item $\Sigma$ is homeomorphic to a connected $2$-dimensional topological manifold.
\item For any~$i$, $F_i$ is contractible and $F_i \cap \Sigma$ is a path.
\item For any $i > j$, $F_j \cap F_i$ is a subset of the endpoints of the path
$F_i \cap \Sigma$.
\item The intersection $O \cap (\Sigma \cup F_1 \cup \cdots F_n)$ is a finite set of points.
\end{enumerate}
If the manifold~$\Sigma$ has negative Euler characteristic, we call $\Delta$ a
\defi{hyperbolic manifold with fins and ornaments} and if the subcomplex~$O$ is
empty, then we call $\Delta$ a \defi{manifold with fins}.
\end{defn}

\begin{thm}\label{thm:main-refined}
If $\Delta$ is a hyperbolic manifold with fins and ornaments, then there is no
degeneration with dual complex~$\Delta$.
\end{thm}

The obstruction to having the dual complex of a degeneration be a hyperbolic
manifold with fins and ornaments is in lifting the dual complex to a tropical
complex. A tropical complex is a $\Delta$-complex, together with the
intersection numbers of the $1$-dimensional strata inside the $2$-dimensional
strata, which are called the structure constants of the tropical
complex~\cite{cartwright-complexes}.

\begin{thm}\label{thm:tropical-complexes}
If $\Delta$ is a hyperbolic manifold with fins and ornaments, then there is no
$2$-dimensional tropical complex with $\Delta$ as its underlying topological
space.
\end{thm}

Note that when we construct the tropical complex from a degeneration, we do not
incorporate valuations from the defining equations, in contrast with both the
suggestion from the introduction of~\cite{cartwright-complexes} and the
construction in~\cite[Ex.~3.10]{gubler-rabinoff-werner}, where edges of the dual
complex have lengths coming from the value group of~$R$. We believe that for
many applications, such metric information will be essential, but in this paper,
we intentionally ignore it for the purpose of being able to use the results
from~\cite{cartwright-surfaces}, where the edges implicitly all have length~$1$.
Philosophically, we think of our approach as taking a deformation in the
category of ``metric tropical complexes'' from the complex which encodes the
valuations to a tropical complex with all edges of length~$1$, where we can
apply Theorem~\ref{thm:tropical-complexes}.

There are two characteristics of the special fiber of a degeneration which are
incorporated into the axioms of a tropical complex. The first is that the
special fiber is principal which gives a relationship among the intersection
numbers with a fixed curve. The second is that Hodge index theorem, which
restricts the possible intersection matrices of a fixed surface in the special
fiber.

Both axioms of a tropical complex are necessary in the proof of the obstruction.
Without the condition coming from the Hodge index theorem, the object would only
be a weak tropical complex, and any $\Delta$-complex lifts to a weak tropical
complex. For example, if $\Delta$ is homeomorphic to a topological manifold,
then choosing all structure constants equal to~$1$ gives a weak tropical
complex, but this will not be a tropical complex if $\chi(\Delta) < 0$, as
explained in Example~\ref{ex:triangulated-manifold}.

On the other hand, Koll\'ar has shown that any finite $n$-dimensional simplicial
complex is realizable as the dual complex of a simple normal crossing
divisor~\cite[Thm.~1]{kollar}, but such a divisor would not give a tropical
complex because the divisor is not necessarily principal. However,
when connected, such a divisor can be realized as the exceptional locus of the
resolution of a normal, isolated singularity~\cite[Thm.~2]{kollar}. Thus, we see
Theorem~\ref{thm:main-refined} as an example of how the global geometry of a
smooth algebraic variety is more restricted than the local geometry of a
singularity, in line with~\cite{kapovich-kollar}.

We also note that unlike the cases in Theorem~\ref{thm:main}, topological
surfaces with non-negative Euler characteristic are all possible as
homeomorphism types of degenerations. In particular, the $2$-sphere, the real
projective plane, the torus, and the Klein bottle appear as degenerations of K3
surfaces, Enriques surfaces, Abelian surfaces, and bielliptic surfaces
respectively. In fact, a partial converse is possible in that the dual complexes
of such degenerations have been classified by results of Kulikov, Persson,
Pinkham, and Morrison~\cite{kulikov,persson,persson-pinkham,morrison}. Note that
topological surfaces of non-negative Euler characteristic all arose from
degenerations of varieties of Kodaira dimension~$0$. However, these
classification results would already suffice to prove Theorem~\ref{thm:main} if
we assumed that the general fiber had Kodaira dimension~$0$. 

\section{Tropical complexes and tropical surfaces}

We begin by recalling the definition of tropical complexes, as introduced
in~\cite{cartwright-complexes} and their properties, as studied
in~\cite{cartwright-surfaces}.  Unlike those papers, we additionally assume that
the underlying $\Delta$-complex is regular, meaning that the faces of any fixed
simplex are distinct. All of our combinatorial results also hold without a
regularity assumption, but the dual complex of a strictly semistable
degeneration always is always regular, so that case is sufficient for our
applications. In addition, we will only work with 2-dimensional tropical
complexes, which we will call tropical surfaces. We will refer to the simplices
of dimensions $0$, $1$, and~$2$ in a $2$-dimensional $\Delta$-complex as its
\defi{vertices}, \defi{edges}, and \defi{facets}, respectively.

\begin{defn}
A \defi{weak tropical surface} is a finite, connected, regular $\Delta$-complex
whose cells have dimension at most~$2$, together with integers $\alpha(v,e)$ for
every endpoint~$v$ of an edge~$e$, such that for each edge~$e$, we have an
equality:
\begin{equation}\label{eq:constraint}
\alpha(v, e) + \alpha(w, e) = \degree(e),
\end{equation}
where $v$ and~$w$ are the endpoints of~$e$ and $\degree(e)$ is the number of
$2$-dimensional facets containing~$e$.

At a vertex~$v$ of a weak tropical surface~$\Delta$, the \defi{local
intersection matrix} $M_v$ is a symmetric matrix whose rows and columns are
indexed by edges containing~$v$ and such that the entry corresponding to
edges~$e$ and~$e'$ is:
\begin{equation}\label{eq:local-intersection-matrix}
(M_{v})_{e, e'} = \begin{cases}
\#\{\mbox{facets containing both $e$ and~$e'$}\} & \mbox{if } e \neq e' \\
- \alpha(w, e) & \mbox{if } e = e',
\end{cases}
\end{equation}
where $w$ is the endpoint of~$e$ other than~$v$. A \defi{tropical surface} is a
weak tropical surface~$\Delta$ such that for every vertex~$v$ of~$\Delta$, $M_v$
has exactly one positive eigenvalue.
\end{defn}

\begin{ex}\label{ex:triangulated-manifold}
Let $\Delta$ be a triangulated manifold, meaning a regular $\Delta$-complex
which is homeomorphic to a $2$-dimensional, connected topological manifold.
Then, every edge~$e$ is contained in two facets, so a symmetric choice for the
structure constants satisfying the constraint~(\ref{eq:constraint}) on~$e$ is to
set $\alpha(v,e) = \alpha(w,e) = 1$ for both endpoints $v$ and~$w$. This gives
us a weak tropical surface.

At any vertex~$v$, the link is a cycle and so, for example, if this cycle has
length~5, the local intersection matrix is:
\begin{equation*}
M_v = \begin{pmatrix}
-1 & 1 & 0 & 0 & 1 \\
1 & -1 & 1 & 0 & 0 \\
0 & 1 & -1 & 1 & 0 \\
0 & 0 & 1 & -1 & 1 \\
1 & 0 & 0 & 1 & -1
\end{pmatrix}
\end{equation*}
An $m \times m$ matrix of this form always has a positive eigenvalue of~$1$ for
the eigenvector $(1, \ldots, 1)^T$. However, if $m > 6$, then there exist at
least two other positive eigenvalues. Therefore, $\Delta$ is a tropical complex
if and only if each vertex is contained in at most $6$ edges. A simple counting
argument shows that if each vertex is contained in at most $6$ edges, then the
Euler characteristic $\chi(\Delta)$ is non-negative. Thus, we've verified the
special case of Theorem~\ref{thm:tropical-complexes} that there is no tropical
surface with all structure constants equal to~$1$ for which the underlying
topological space is a hyperbolic manifold.
\end{ex}

In this paper, the main purpose of the structure constants $\alpha(v,e)$ is to
define a sheaf of linear functions on a weak tropical surface~$\Delta$. These
linear functions are the tropical analogue of non-vanishing regular functions in
algebraic geometry.

\begin{defn}[Const.~4.2 and Def.~4.3 in \cite{cartwright-complexes}]
\label{def:linear}
Let $\Delta$ be a weak tropical surface and $e$ an edge of $\Delta$. We define
$N_e$ to be the simplicial complex consisting of~$e$ with a $2$-dimensional
simplex attached for each facet of~$\Delta$ containing~$e$. Then, there exists
a natural map $\pi_e \colon N_e \rightarrow \Delta$ which is an open inclusion
on $N_e^o$, which we define to be the union of interior of~$e$ and of the
interiors of the facets of~$N_e$.

We also define a map $\phi_e \colon N_e \rightarrow \RR^{d+2} / \RR$, where $d$
is the number of facets of~$\Delta$ containing~$e$ and the quotient is by the
line in $\RR^{d+2}$ generated by the vector $(1, \ldots, 1, -\alpha(v, e),
-\alpha(w,e))$. The map $\phi_e$ sends each vertex of $N_e$ to the
image~$\mathbf e_i$ of the $i$th basis vector of $\RR^{d+2}$, with $v$ and $w$
going to $\mathbf e_{d+1}$ and $\mathbf e_{d+2}$ respectively. We then extend
$\phi_e$ linearly to all of $N_e$. The vectors $\mathbf e_1, \ldots, \mathbf
e_{d+2}$ generate a lattice inside $\RR^{d+2} / \RR$ and we say that a function
$\ell \colon \RR^{d+2} / \RR \rightarrow \RR$ has linear slopes if $\ell(\mathbf
e_i) - \ell(\mathbf e_j)$ is an integer for any two vectors $\mathbf e_i$ and
$\mathbf e_j$.

Finally, we say that a continuous function~$\phi$ on an open subset~$U \subset
\Delta$ is \defi{linear} if the following two conditions hold. First, if we
identify the interior of any facet~$f$ meeting~$U$ with a unimodular simplex
in~$\RR^2$, then $\phi$ is an affine linear function with integral slopes on the
interior of~$f$. Second, on any edge $e$ meeting $U$, $\phi \vert_{U \cap
N_{e}^o} = \ell \circ \phi_e$ for some linear function $\ell \colon \RR^{d+2} /
\RR \rightarrow \RR$ with integral slopes. We write $\cA$ for the sheaf of
linear functions on~$\Delta$.
\end{defn}

Note that the image of the map $\phi_e$ in Definition~\ref{def:linear} lies in
the affine linear space consisting of the linear combinations $\sum c_i \mathbf
e_i$ such that $c_1 + \cdots + c_{d+2} = 1$, which is a proper subset of
$\RR^{d+2}/\RR$ because of relation~(\ref{eq:constraint}) in the definition of a
weak tropical surface. Roughly speaking, the quotient in
Definition~\ref{def:linear} means that linearity on a weak tropical surface
imposes one condition beyond linearity on each facet of~$\Delta$. More
precisely, we have the following:

\begin{lem}\label{lem:all-but-one}
Let $e$ be an edge of a tropical surface $\Delta$ and let $N_e^o$ be the union
of the interiors of $e$ and of the facets containing~$e$, as in
Definition~\ref{def:linear}. If $\phi$ is a linear function on $N_e^o$, and is
constant on all but one of these facets, then $\phi$ is constant.
\end{lem}

\begin{proof}
If $\phi_e$ is as in Definition~\ref{def:linear}, then
$\phi$ factors as $\ell \circ \phi_e$ for some affine linear function~$\ell$.
We also let $d$ and $\mathbf e$ be as in Definition~\ref{def:linear}, and we can
assume that the facets containing $e$ are numbered so that $\phi$ is constant on
all but the first one. Then, we have the linear relation
\begin{equation*}
\mathbf e_1 = -\mathbf e_2 - \cdots - \mathbf e_{d}
+ \alpha(v, e) \mathbf e_{d+1} + \alpha(w, e) \mathbf e_{d+2},
\end{equation*}
where the sum of the coefficients is~$1$ by the relation~(\ref{eq:constraint}),
and so $\ell(\mathbf e_2) = \cdots = \ell(\mathbf e_{d+2})$ implies that
$\ell(\mathbf e_1) = \ell(\mathbf e_2) = \cdots = \ell(\mathbf e_{d+2})$. Therefore, $\ell$ is constant on
$\phi_e(N_e^o)$, which is what we wanted to show.
\end{proof}

\begin{ex}
As in Example~\ref{ex:triangulated-manifold}, suppose we have a triangulated
manifold with $\alpha(v,e) = 1$ for all endpoints $v$ of all edges~$e$. Then,
in Definition~\ref{def:linear}, the interior of~$e$ and of the two facets
containing~$e$ is identified with the interior of a unit square in~$\RR^2$, with
the square triangulated along a diagonal. Linear functions on this neighborhood
of~$e$ are equivalent to affine linear functions on the square. Thus, in this
case, Lemma~\ref{lem:all-but-one} amounts to the observation that affine linear
functions on~$\RR^2$ are determined by their restriction to on any open set.
\end{ex}

Any constant function on a weak tropical surface is linear, so the sheaf
of locally constant $\RR$-values functions is a subsheaf of~$\cA$. If we denote
the quotient sheaf $\cA/ \RR$ by~$\cD$, then we have a long exact sequence in
sheaf cohomology~\cite[Sec.~3]{cartwright-surfaces}:
\begin{equation}\label{eq:long-exact}
0 \rightarrow H^0(\Delta, \RR) \rightarrow H^0(\Delta, \cA) \rightarrow
H^0(\Delta, \cD) \rightarrow H^1(\Delta, \RR)
\rightarrow \cdots
\end{equation}
One of the main results from~\cite{cartwright-surfaces} is the following:
\begin{thm}[Thm.~4.7 in \cite{cartwright-surfaces}]\label{thm:differentials}
If $\Delta$ is a tropical surface which is locally connected through
codimension~$1$, the $\RR$-span of the image of the morphism $H^0(\Delta, \cD)
\rightarrow H^1(\Delta, \RR)$ has codimension at
most~$1$ in $H^1(\Delta, \RR)$.
\end{thm}

\noindent In Theorem~\ref{thm:differentials}, \defi{locally connected
through codimension~$1$} means that the link of each vertex is connected.

If linear functions on weak tropical surfaces are taken to be analogous to
non-vanishing regular functions, then the analogues of rational functions on an
algebraic variety come from relaxing
the linearity condition in codimension~1. More precisely, suppose that $d_1,
\ldots, d_m$ are closed line segments, each in one facet of a weak
tropical surface~$\Delta$, such that $d_i \cap d_j$ is finite
for distinct $i$ and~$j$. Then we say a function $\phi$ on an open subset $U
\subset \Delta$ is a \defi{piecewise linear function whose divisor is supported
in $d_1 \cup \cdots \cup d_m$} if $\phi$ is continuous and $\phi$ is linear on
$U \setminus (d_1 \cup \cdots \cup d_m)$. Although we will not need it in this
paper, we can justify our terminology with:
\begin{prop}[Prop.~4.5 in \cite{cartwright-complexes}]\label{prop:divisor}
Let $d_1, \ldots, d_m$ be line segments in a weak tropical surface~$\Delta$ as
above. If $U$ is an open set meeting all of the $d_i$, then there exists a
homomorphism from the group of piecewise linear functions~$\phi$ whose divisor
is supported on $d_1 \cup \cdots \cup d_m$ under addition to formal sums of the
$d_i$, known as the divisor of~$\phi$. The maximal open subset of~$U$ on which
such a function~$\phi$ is linear is the complement of those $d_i$ with non-zero
coefficient in the divisor of~$\phi$.
\end{prop}

Finally, we have an analogue of the maximum modulus principle from complex
analysis. The result in \cite{cartwright-surfaces} also applies to piecewise
linear functions whose divisor has non-negative coefficients, but we'll
only need it for linear functions, i.e.\ when the divisor is trivial.

\begin{prop}[Prop.~2.11 in \cite{cartwright-surfaces}]
\label{prop:maximum-modulus}
Let $\Delta$ be a tropical surface which is connected through codimension~$1$.
If $\phi$ is a linear function on a connected open set~$U$ which achieves its
maximum on~$U$, then $\phi$ is constant.
\end{prop}

\section{Degenerations}

In this section, we construct weak tropical surfaces from strictly semistable
degenerations. The case of regular semistable degenerations over discrete
valuation rings was treated in~\cite[Sec.~2]{cartwright-complexes}, but here we
want to work over possibly non-discrete valuation rings. The data of a weak
tropical surface only depends on the special fiber as a simple normal crossing
scheme over the residue field, and not on the valuation ring, so we can use the
same construction as~\cite[Sec.~2]{cartwright-complexes}, which we now recall.

We let $\X$ be a strictly semistable degeneration over~$R$ and, as stated in the
introduction, the dual complex~$\Delta$ has one $k$-dimensional simplex for each
stratum of dimension $2-k$ in the special fiber of~$\X$. Thus, if $e$
is an edge with endpoints $v$ and~$w$, then we let $C_e$ and $C_w$ denote the
curve and surface corresponding to $e$ and~$w$ respectively. We set $\alpha(v,e)
= - C_e^2$, where $C_e^2$ denotes the self-intersection number of~$C_e$ in
$C_w$. Even for regular strictly semistable degenerations over discrete
valuation rings, this data may only give a weak tropical surface without an
additional technical condition of robustness in dimension~2
\cite[Prop.~2.7]{cartwright-complexes}. Over non-discrete valuation rings, we
also get weak tropical surfaces

\begin{prop}\label{prop:degeneration}
The special fiber of any degeneration~$\X$ yields a weak tropical
surface~$\Delta$ such that the local intersection matrix~$M_v$ has at most one
positive eigenvalue for each vertex~$v$ of~$\Delta$.
\end{prop}

\begin{proof}
For $\Delta$ to be a weak tropical surface, we need to check that for any
edge~$e$ with endpoints $v$ and~$w$, we have the equality (\ref{eq:constraint}):
\begin{equation}\label{eq:constraint-prop}
\alpha(v,e) + \alpha(w, e) = \degree e.
\end{equation}
Let $C_e$ be the curve corresponding to~$e$ in the special fiber of~$\X$, and 
by our semistability condition, on a Zariski open neighborhood meeting
$C_e$, there is an
\'etale map to $\Spec R[x,y,z]/\langle x y - \pi \rangle$ for some element~$\pi$
in the maximal ideal of~$R$. Then, the principal Cartier divisor defined by
$\pi$ can be written, at least in a formal open neighborhood of~$C_e$, as the union of
Cartier divisors, each of which is supported on an irreducible component of the
special fiber of~$\X$. For example, in the above chart, the functions~$x$
and~$y$ pull back to give defining equations for each of the components
containing~$C_e$.

Thus, using linearity of the intersection product~\cite[Prop.~5.9(b)]{gubler},
we can split up the intersection of the principal divisor defined by~$\pi$ with
the curve $C_e$ into terms coming from the components of the special fiber
of~$\X$. For components of the special fiber which don't contain~$C_e$, if we
pull back to~$C_e$ we get a Cartier divisor equal to the points of intersection,
with multiplicities equal to~$1$. Thus, the degree of the intersection of such a
Cartier divisor with~$C_e$ is equal to the number of points of intersection by
the projection formula~\cite[Prop.~5.9(c)]{gubler}. The total degree for all
components which don't contain~$C_e$ gives $\deg e$, which is the right-hand
side of~(\ref{eq:constraint-prop}).

Now consider the two components~$C_v$ and~$C_w$ containing~$C_e$. If we pull
back the Cartier divisor supported on~$C_v$ to $C_w$ then we get the
divisor~$C_e$ on~$C_w$. The self-intersection of~$C_e$ is $-\alpha(v, e)$ by the
definition of the structure constants. Thus, using the projection formula again,
the components containing~$C_e$ contribute a cycle of degree equal to
$-\alpha(v, e) - \alpha(w, e)$, so the desired
equality~(\ref{eq:constraint-prop}) follows because $\pi$ obviously defines a
principal divisor.

Finally, we claim that the local intersection matrix~$M_v$ records the
intersection theory on the surface of the special fiber corresponding to~$v$,
restricted to curves of the special fiber. For the diagonal entries of the local
intersection matrix~(\ref{eq:local-intersection-matrix}), this follows
immediately from our definition of the structure constants. The off-diagonal
entries of the local intersection matrix count facets containing two edges~$e$
and~$e'$, which are in bijection with the number of reduced points in the
intersection of corresponding curves $C_e$ and~$C_{e'}$, and thus equal to the
intersection number $C_e \cdot C_{e'}$. Therefore, by the Hodge index theorem,
the local intersection matrix~$M_v$ can have at most one positive eigenvalue.
\end{proof}

One approach to obtaining a tropical surface instead of a weak tropical surface
is Proposition~2.10 in \cite{cartwright-complexes}, which shows that for
degenerations with projective components, robustness can be obtained by
appropriate blow-ups. Rather than adapting this proposition to the case of
non-discrete valuations, while also keeping track of the effect on the
underlying topological space, it is more convenient to perform the modification
combinatorially:

\begin{lem}\label{lem:combinatorial-blow-up}
Let $\Delta$ be a weak tropical surface and suppose that for
each vertex~$v$ of~$\Delta$, the local intersection matrix~$M_v$ has at most one
positive eigenvalue. Then, there exists a
tropical surface~$\Delta'$ such that the underlying topological space of
$\Delta'$ is formed by attaching a finite number of $2$-simplices to edges
of~$\Delta$.
\end{lem}

\begin{proof}
We suppose that $v$ is a vertex of~$\Delta$ such that $M_v$ has no positive
eigenvalues, i.e.\ it is negative semidefinite. Let $e$ be an edge
containing~$v$ and let $w$ be
the other endpoint of~$e$. We attach an additional $2$-simplex onto $e$ and
label the new vertex $u'$, with the new edges $e_v'$ and $e_w'$. We use $v'$,
$w'$ and $e'$ to denote the representatives of $v$, $w$, and $e$ in the new weak
tropical surface $\Delta'$. We assign the coefficients on~$\Delta'$ to be the
same as on~$\Delta$, except that:
\begin{align*}
&\alpha(w', e') = \alpha(w, e) &
&\alpha(v', e') = \alpha(v, e) + 1 \\
&\alpha(w', e_w') = 0 &
&\alpha(u', e_w') = 1 \\
&\alpha(v', e_v') = 2 &
&\alpha(u', e_v') = -1
\end{align*}
Then we claim the local intersection matrices $M_{u'}$ and $M_{v'}$ each have
exactly one positive eigenvalue and that the number of positive eigenvalues
of~$M_{w'}$ is the same as that of~$M_w$. Once we show the claim, then we can
repeat the above construction at each vertex~$v$ whose local intersection
matrix~$M_v$ is negative semidefinite to get the desired tropical surface.

The first part of the claim is that the local intersection matrices
\begin{equation*}
M_{u'} = \begin{pmatrix}
0 & 1 \\
1 & -2
\end{pmatrix}
\mbox{ and }
M_{v'} = \begin{pmatrix}
* & \cdots & * & 0 \\
\vdots & & \vdots & \vdots \\
* & \cdots & -\alpha(w, e)  & 1 \\
0 & \cdots & 1 &  1 
\end{pmatrix}
\end{equation*}
have exactly one positive eigenvalue each, where $*$ indicates the parts of
$M_{v'}$ that coincide with~$M_v$. For $M_{u'}$, there is exactly one positive
eigenvalue because
it has determinant $-1$. For $M_{v'}$, we use the change of  coordinates:
\begin{equation*}
NM_{v'}N^T = \begin{pmatrix}
* & \cdots  & * & 0 \\
\vdots & & \vdots & \vdots \\
* & \cdots & -\alpha(w, e) - 1 & 0 \\
0 & \cdots & 0 & 1
\end{pmatrix}
\mbox{ where }
N = \begin{pmatrix}
1 & \cdots  & 0 & 0 \\
\vdots & \ddots & \vdots & \vdots \\
0 & \cdots & 1 & -1 \\
0 & \cdots & 0 & 1
\end{pmatrix}
\end{equation*}
to get a block diagonal matrix whose upper left block is $M_v$ minus a negative
semidefinite diagonal matrix, and is thus negative semidefinite. Moreover, the
lower right block of $N M_{v'} N^T$ is a single positive entry, so $N M_{v'}
N^T$ has exactly one positive eigenvalue, as in the first part of the claim.

For the second part of our claim, we want to show that
\begin{equation*}
M_{w'} = \begin{pmatrix}
* & \cdots & * & 0 \\
\vdots & & \vdots & \vdots \\
* & \cdots & -\alpha(v, e) - 1 & 1 \\
0 & \cdots & 1 & -1
\end{pmatrix}
\end{equation*}
has the same number of positive eigenvalues as $M_w$, where $*$ denotes parts of
the matrix which coincide with $M_w$. We again use a change of coordiates to a
block diagonal matrix:
\begin{equation*}
P M_{w'} P^T = \begin{pmatrix}
* & \cdots & * & 0 \\
\vdots & & \vdots & \vdots \\
* & \cdots & -\alpha(v, e) & 0 \\
0 & \cdots & 0 & -1
\end{pmatrix}
\mbox{ where }
P = \begin{pmatrix}
1 & \cdots & 0 & 0 \\
\vdots & \ddots & \vdots & \vdots \\
0 & \cdots & 1 & 1 \\
0 & \cdots & 0 & 1
\end{pmatrix}.
\end{equation*}
The upper left block of $P M_{w'} P^T$ agrees with $M_w$ and the lower right
block adds a single negative eigenvalue. Thus, $M_{w'}$ has the same number of
positive eigenvalues as $M_w$, which completes the proof of the claim.
\end{proof}

\section{Proof of the main theorems}

The crux of Theorem~\ref{thm:tropical-complexes} and thus of
Theorem~\ref{thm:main-refined} is the following lemma:

\begin{lem}\label{lem:restriction}
Let $\Delta$ be a tropical surface whose underlying $\Delta$-complex is a
manifold with fins and is connected through codimension~$1$. If $s$ is a facet
contained in the manifold subcomplex of~$\Delta$, and $U_s$ denotes the interior
of~$s$, then the restriction map
\begin{equation*}
H^0(\Delta, \cD) \rightarrow H^0(U_s, \cD) \isom \ZZ^2
\end{equation*}
is injective.
\end{lem}

\begin{proof}
Note that the isomorphism $H^0(U_s, \cD) \isom \ZZ^2$ holds because affine
linear functions on $U_s$ are equivalent to affine linear functions with
integral slopes on a unimodular simplex in $\RR^2$ by definition. Thus, $\cA$
restricted to $U_s$ is isomorphic to the locally constant sheaf with values in
$\RR \times \ZZ^2$, and the quotient sheaf~$\cD = \cA / \RR$ is isomorphic to
$\ZZ^2$.

Now, we let $\Sigma$ and $F_1, \ldots, F_n$ denote the manifold and fins of the
simplicial complex, as in Definition~\ref{def:manifold-fins-ornaments}. We
suppose $\omega$ a global section of~$\cD$ such that the restriction of~$\omega$
to $U_s$ is trivial, and we want to show that $\omega$ is trivial. We start with
$V = U_s$ and then we'll expand the open set~$V$ until it is all of~$\Delta$. At
each step, $V$ will either be disjoint from each fin $F_i$ or contain~$F_i$,
except possibly the endpoints of $F_i \cap \Sigma$. In particular, the
boundary of~$V$ will be contained in~$\Sigma$.

First suppose that there exists an edge~$e$ in the boundary of~$V$ such that $e$
is not contained in any of the fins. Let $f$ denote the $2$-simplex
bordering~$e$ whose interior is in~$V$ and let $f'$ denote the $2$-simplex on
the other side of~$e$. Then, Definition~\ref{def:linear} 
identifies the union of the interiors of $f$,
$f'$, and~$e$ with an open subset of~$\RR^2$, and the sections of $\cA$ are
exactly affine linear functions with integral slope on this set. As above, $\cA$
and~$\cD$ are therefore locally constant sheaves with values in $\RR \times
\ZZ^2$ and $\ZZ^2$ respectively. Thus, we can expand~$V$ to include the
interiors of $e$ and~$f'$, where $\omega$ is also zero.

Second, we assume that every edge in the boundary of~$V$ is contained in some
$\Sigma \cap F_i$. Let $i$ be the maximal index such that $F_i$ intersects the
boundary of~$V$. Then, the entire path $\Sigma \cap F_i$ must be in the boundary
of~$V$ or else there would be a fin $F_j$ with $j < i$ intersecting $F_i$ not at
its endpoint, which would contradict
Definition~\ref{def:manifold-fins-ornaments}. In particular, $\omega$ must
vanish along $\Sigma \cap F_i$.

Let $\widetilde\cA_i$ be the sheaf of piecewise linear functions on~$\Delta$
whose divisors are supported on $\Sigma \cap F_i$. If we let $\widetilde\cD_i$
denote the quotient sheaf $\widetilde\cA_i / \RR$, then we can give a global
section~$\widetilde\omega_i$ of $\widetilde\cD_i$ defined piecewise such
that $\widetilde \omega_i$ vanishes on $\Delta \setminus F_i$ and it agrees
with $\omega$ on $F_i \setminus \Sigma$. This defines a valid section
of~$\widetilde \cD_i$ because, as we noted, $\omega$ vanishes along $\Sigma \cap
F_i$, and $\widetilde\omega_i$ is clearly a section of $\cD$ away from $\Sigma
\cap F_i$, on which we only require continuity. Consider the long exact sequence
of cohomology associated to the quotient~$\widetilde\cD_i$, analogous
to~(\ref{eq:long-exact}):
\begin{equation*}
0 \rightarrow H^0(\Delta, \RR) \rightarrow
H^0(\Delta, \widetilde\cA_i) \rightarrow
H^0(\Delta, \widetilde\cD_i) \rightarrow H^1(\Delta, \RR) \rightarrow
\end{equation*}
Since $\widetilde\omega_i$ is only non-trivial on~$F_i$,
which is contractible, the image of $\widetilde\omega_i$ in $H^1(\Delta, \RR)$
is trivial, so $\widetilde\omega_i$ lifts to an element of $H^0(\Delta,
\widetilde\cA_i)$, which we also denote by $\widetilde\omega_i$ and we choose
the representative such that $\widetilde\omega_i$ is zero on~$\Sigma$.

If $\widetilde\omega_i$ is non-constant, then it must have a maximum value
strictly greater than zero or minimum value strictly less than zero. Then, it
would have its maximum or minimum, respectively, on $F_i \setminus \Sigma$. We
apply Proposition~\ref{prop:maximum-modulus} to the linear function
$\widetilde\omega_i\vert_{F_i \setminus \Sigma}$ or to its negative to show that
$\widetilde\omega_i\vert_{F_i \setminus \Sigma}$ must be constant, and
so $\widetilde\omega_i$ is zero everywhere. Thus, $\omega$ is identically zero
on~$F_i$, and so we can expand~$V$ to include $F_i \setminus \Sigma$.

We've now shown that for each edge~$e$ of $\Sigma \cap F_i$, the
section~$\omega$ is zero on all but one simplex containing~$e$, namely the
simplex in~$\Sigma$ on the other side from~$V$. As in
Definition~\ref{def:linear}, we let $N_e^o$ denote the union interior of~$e$
with the interiors of the facets containing~$e$. Since $N_e^o$ is simply
connected, we can lift $\omega \vert_{N_e^o}$ to a linear function~$\phi$ on
$N_e^o$. Then, by Lemma~\ref{lem:all-but-one}, $\phi$ is constant on $N_e^o$,
and so $\omega$ vanishes on $N_e^o$. Therefore, we can further expand~$V$ to
also include the interiors of all $2$-simplices meeting $F_i$ and the path
$\Sigma \cap F_i$.

At the end, we will have that $\omega$ is zero on an open set~$V$ which contains
the interior of every $2$-simplex in~$\Delta$ and since affine linear functions
are continuous by definition, this means that $\omega$ is zero, which finishes
the proof of the lemma.
\end{proof}

We use Lemma~\ref{lem:restriction} to prove the following strengthening of
Theorem~\ref{thm:tropical-complexes}.

\begin{thm}\label{thm:weak-tropical-complexes}
If $\Delta$ is a hyperbolic manifold with fins and ornaments, then there is no
weak tropical surface, with $\Delta$ as its underlying topological space, and
such that for every vertex~$v$ of~$\Delta$, $M_v$ has at most one positive
eigenvalue.
\end{thm}

\begin{proof}
Suppose that $\Delta$ is a
weak tropical surface whose underlying $\Delta$-complex is as in the theorem
statement.
We can assume that when decomposing $\Delta$ as in
Definition~\ref{def:manifold-fins-ornaments}, the subcomplex of
ornaments~$O$ is maximal, so that if we let $\Delta'$ denote the subcomplex
consisting of just
the manifold and fins, then $\Delta'$ is locally
connected through codimension~$1$. Then, taking the restriction of
the structure constants from~$\Delta$, we get that $\Delta'$ has the structure
of a weak tropical surface, because the local structure around each edge
of~$\Delta'$ is unchanged.
Moreover, at each vertex~$v$ of~$\Delta'$, the local intersection matrix~$M'_v$
is a block of the block diagonal matrix~$M_v$
for~$\Delta$. Therefore, $M'_v$ also has at most one positive eigenvalue. 

Next, we apply Lemma~\ref{lem:combinatorial-blow-up} to transform~$\Delta'$ into
a tropical surface $\Delta''$ by gluing
simplices onto edges of~$\Delta'$. Whenever we glue a simplex onto an edge~$e$
which is contained in one of the fins~$F_i \subset \Delta'$, we can include that
simplex in the fin, which remains contractible and its intersection with the
manifold~$\Sigma$ is unchanged. If we glue a simplex onto an edge~$e$ contained
in the manifold~$\Sigma$, then the simplex forms a new fin~$F_{n+1}$, numbered
after all the other fins. Since $F_{n+1} \cap \Sigma$ is a single edge, the
intersection of~$F_{n+1}$ with any other fin will be a subset of the endpoints
of this edge. Thus, $\Delta''$ is still a hyperbolic manifold with fins.

Finally, suppose that the manifold $\Sigma \subset \Delta''$ is not
orientable. Then $\Sigma$ has a $2$-to-$1$ orientable cover,
corresponding to an index~$2$ subgroup of its
fundamental group~\cite[Prop.~3.25]{hatcher}. Since $\Sigma
\subset \Delta''$ is a homotopy equivalence, the oriented cover extends to a
cover of~$\Delta''$, which we call $\widetilde \Delta''$. Each fin~$F_i$ of
$\Delta''$ is
attached along a path of~$\Sigma$, and so the preimage of~$F_i$ in $\widetilde
\Delta''$ is two disjoint fins, which we number $F_{2i-1}$ and $F_{2i}$ to get a
manifold with fins. Moreover, the Euler characteristic is multiplicative when
taking covers, so
$\chi(\widetilde \Delta'')$ is again negative. Therefore, we can replace $\Delta''$ with
$\widetilde \Delta''$ and so from now on we assume that the manifold $\Sigma
\subset \Delta''$ is orientable.

By the classification of compact topological
surfaces~\cite[Thm.~6.3]{gallier-xu}, the Euler characteristic of a compact
oriented surface is even, and so
$\chi(\Delta'') \leq -2$.
Therefore,
$\dim_{\RR} H^1(\Delta'', \RR) =
2 - \chi(\Delta'') \geq 4$.
By Theorem~\ref{thm:differentials}, the $\RR$-span of the image of
$H^0(\Delta'', \cD)$ has codimension at most 1 in $H^0(\Delta'', \RR)$, so the
rank of $H^0(\Delta'', \cD)$ as an Abelian group is at least $\dim H^1(\Delta'',
\RR) - 1 \geq 3$. On the other hand, by Lemma~\ref{lem:restriction},
$H^0(\Delta'', \cD)$ is a subgroup of $\ZZ^2$, and so a free Abelian group of
rank at most~$2$. Therefore, we have a contradiction, so the weak tropical
surface~$\Delta$ cannot exist.
\end{proof}

\begin{proof}[Proof of Theorem~\ref{thm:main-refined}]
Suppose $\X$ is a strict semistable degeneration whose dual complex~$\Delta$ is
a hyperbolic manifold with fins and ornaments. Then, by
Proposition~\ref{prop:degeneration}, $\Delta$ has the structure of a weak
tropical complex such that the matrix $M_v$ has at most one positive eigenvalue
for every vertex~$v$. However, by Theorem~\ref{thm:weak-tropical-complexes},
such a weak tropical complex cannot exist, so we conclude that $\X$ cannot
exist.
\end{proof}

\end{document}